\documentclass[12pt, reqno]{amsart}
\usepackage{microtype}
\usepackage{float}
\usepackage[dvipdfmx]{graphicx}
\usepackage[usenames]{color}
\usepackage{hyperref}
\usepackage{bbm}
\usepackage{bm}
\usepackage[all]{xy}
\usepackage{mathrsfs}
\usepackage{amsmath,amsfonts,amssymb,amsthm,amscd}
\usepackage[subrefformat=parens]{subcaption}
\usepackage{xcolor}
\usepackage{comment}
\usepackage{tikz}
\usetikzlibrary{arrows.meta, shapes, trees, calc}

\makeatletter
\numberwithin{equation}{section}

\newcommand{\m}{\mu}

\renewcommand{\SS}{\Sigma}

\newcommand{\e}{\varepsilon}
\newcommand{\f}{\varphi}

\newcommand{\N}{{\mathbb N}}
\newcommand{\R}{{\mathbb R}}

\renewcommand{\Pr}{{\mathbb P}}

\DeclareMathOperator{\id}{{\rm id}}

\renewcommand{\(}{\left(}
\renewcommand{\)}{\right)}

\newcommand{\wbar}{\overline}

\newcommand{\Pb}{{\bf P}}
\newcommand{\Eb}{{\bf E}\,}

\newcommand{\TV}{{\rm TV}}
\newcommand{\Var}{{\rm Var}}
\newcommand{\Cov}{{\rm Cov}}

\newcommand{\diag}{{\rm diag}}

\newcommand{\mb}[1]{\bm{#1}}

\newcommand{\cU}{\mathcal{U}}

\newtheorem{theorem}{Theorem}[section]
\newtheorem{proposition}[theorem]{Proposition}
\newtheorem{lemma}[theorem]{Lemma}
\newtheorem{corollary}[theorem]{Corollary}

\theoremstyle{definition}

\theoremstyle{remark}
\newtheorem{remark}[theorem]{Remark}

\begin{document}

\title[Noise stability on hyperbolic groups]
      {Noise stability on hyperbolic groups}
\author{Timoth\'ee B\'enard}
\address{LAGA - Universit\'e Sorbonne Paris Nord, Villetaneuse, FRANCE
}
\email{benard@math.univ-paris13.fr}

\author{Ryokichi Tanaka}
\address{Department of Mathematics, 
Kyoto University, Kyoto, JAPAN}
\email{rtanaka@math.kyoto-u.ac.jp}
\thanks{R.T.\ is partially supported by 
JSPS Grant-in-Aid for Scientific Research JP24K06711}
\date{\today}

\maketitle

\begin{abstract}
We show that symmetric random walks on non-elementary hyperbolic groups with non-zero homomorphisms into the reals are noise stable at linear scale under finite exponential moment condition.
\end{abstract}

\section{Introduction}

Let $\Gamma$ be a countable group, let $\mu$ be a probability measure on it. 
A $\mu$-{\bf random walk}  on $\Gamma$ starting at the identity $\id$ is a sequence of random variables $(w_{n})_{n\geq 0}$ of the form 
 $w_n:=\gamma_1\cdots \gamma_n$ (with $w_0:=\id$) where the $\gamma_{i}$'s are i.i.d.\ variables of law $\mu$.
The {\bf noise sensitivity and stability problem} for a $\mu$-random walk asks the following.
For each fixed real $\rho \in (0, 1)$,
let $w_n^\rho:=\gamma_1^\rho\cdots \gamma_n^\rho$
where $\gamma_i^\rho$ is resampled according to $\mu$ with probability $\rho$,
or retained as $\gamma_i$ with probability $1-\rho$,
independently in every step.
Does the pair $(w_n, w_n^\rho)$ become independent as $n$ tends to infinity?

More precisely,
for each real $\rho \in [0, 1]$,
consider  on $\Gamma \times \Gamma$ the probability measure
\[
\pi^\rho:=\rho \mu \otimes \mu+(1-\rho)\mu_\diag,
\]
where $\mu\otimes \mu$ is the product measure and $\mu_\diag((\gamma, \gamma'))=\mu(\gamma)$ if $\gamma=\gamma'$ and $0$ else.
We write $\nu_n:=\nu^{\ast n}$ for the $n$-fold convolution of a probability measure $\nu$ on a group. 
Note that the $n$-step distribution of a  $\pi^\rho$-random walk starting at $\id$ is given by $\pi^\rho_n$,
which coincides with the law of $(w_n, w^\rho_n)$ on $\Gamma \times \Gamma$.
Moreover, $w_n$ and $w^\rho_n$ have the same distribution $\m_n$ on $\Gamma$.  
The $\mu$-random walk on $\Gamma$ is called {\bf noise sensitive} in total variation if for every $\rho \in (0, 1]$, 
\[
\|\pi^\rho_n-\mu_n\otimes \mu_n\|_\TV \to 0 \quad \text{as $n \to \infty$}.
\]
In the above the total variation distance is defined  for probability measures $\nu^{(i)}$, $i=1, 2$ by
\[
\|\nu^{(1)}-\nu^{(2)}\|_\TV:=\sup_{A \subset \Gamma \times \Gamma}|\nu^{(1)}(A)-\nu^{(2)}(A)|,
\]
see also Remark \ref{Rem:TV}.
In contrast, the $\mu$-random walk on $\Gamma$ is called {\bf noise stable} in total variation if  for every $\rho \in [0, 1)$,
\[
\|\pi^\rho_n-\mu_n\otimes \mu_n\|_\TV \to 1 \quad \text{as $n \to \infty$}.
\]

Informally, on the one hand, if a $\mu$-random walk is noise sensitive,
then even though 
only a small number of increments are resampled (say, $\rho=0.01$),
the pair $(w_n, w^\rho_n)$ becomes independent as $n$ tends to infinity.
On the other hand,
if a $\mu$-random walk is noise stable,
then even though 
a large number of increments are resampled (say, $\rho=0.99$),
the pair $(w_n, w^\rho_n)$ is still distinguishable from an independent pair.
Note that $w_n$ and $w^\rho_n$ are independent when $\rho=1$.

The study of noise sensitivity for random walks on groups was proposed by Benjamini and Brieussel in \cite{BenjaminiBrieussel}. 
They observe that if $\Gamma$ is a finite group, then every $\mu$-random walk is noise sensitive \cite[Proposition 5.1]{BenjaminiBrieussel}. They also prove noise sensitivity for some (generating) $\mu$-random walk on the infinite dihedral group \cite[Theorem 1.4]{BenjaminiBrieussel}. Later, this result has been  generalized in \cite{Tnss}.
It is still a challenge to understand which finitely generated infinite groups admit a noise sensitive random walk. See also related questions in  \cite{BenjaminiBrieussel} and \cite[Section 3.3.4]{KalaiICM2018}.

The noise stability is a strong negation of noise sensitivity.
It has been observed that standard random walks on integer lattices,
i.e.\ non trivial free abelian groups of finite rank are neither noise sensitive nor noise stable \cite[Appendix A]{Tnss}.
An inspection shows that the simple random walk on a free semigroup of rank at least two is noise stable.
In a more general setting,
the following has been shown for every word hyperbolic group:
If a probability measure $\mu$ is non-elementary  (cf.\ Section \ref{Sec:limit_laws}) with finite first moment,
then there exists  $\rho_0 \in (0, 1]$ such that 
$\|\pi^\rho_n-\mu_n\otimes \mu_n\|_\TV \to 1$ as $n \to \infty$ for all $\rho \in [0, \rho_0)$ 
\cite[Theorem 1.3]{NonNS}.
It is not known as to whether $\rho_0=1$ in this generality.
We show that this is the case for a certain class of word hyperbolic groups and probability measures.

\begin{theorem}\label{Thm:intro}
Let $\Gamma$ be a non-elementary word hyperbolic group which admits a non-zero homomorphism to $\R$. Let $\mu$ be a symmetric probability measure on $\Gamma$ with finite exponential moment and whose support is not included in a proper subgroup of $\Gamma$.
It holds that for every $\rho \in [0, 1)$,
\[
\|\pi^\rho_n-\mu_n\otimes\mu_n\|_\TV \to 1 \quad \text{as $n \to \infty$},
\]
i.e.\ the $\mu$-random walk on $\Gamma$ is noise stable in total variation.
\end{theorem}

Recall that a probability measure $\mu$ is called symmetric if $\mu(\gamma^{-1})=\mu(\gamma)$ for all $\gamma \in \Gamma$.
Moreover, as a hyperbolic group is finitely generated,  the condition that  $\Gamma$ admits a non-zero homomorphism to $\R$ is equivalent to the requirement that its  abelianization $\Gamma/[\Gamma, \Gamma]$ is infinite.  Examples include finitely generated non-abelian free groups  and fundamental groups of closed Riemann surfaces different from the sphere and the torus.

Let us point out that Theorem \ref{Thm:intro} is already new for the simple random walk on a free group -- an explicit form of the distribution $\pi^\rho_n$ does not seem available even in this special case.
We show Theorem \ref{Thm:intro} in a stronger form Theorem \ref{Thm:harmonic}, stating that the mass distributions $\pi^\rho_n$ and $\mu_n\otimes\mu_n$ not only do not intersect much, but in fact separate at linear scale. For this, we fix a word metric on $\Gamma$, and set  $\lambda_{\mu}>0$ the associated escape rate of the $\mu$-walk on $\Gamma$, see Section \ref{Sec:limit_laws}.
We equip $\Gamma^2$  with the $l_{\infty}$-metric with respect to $\Gamma$-coordinates, i.e. $d(g,h)=\max\{d(g_{1},h_{1}), d(g_{2},h_{2})\}$ for $g,h\in \Gamma^2$.
We consider  a relaxed notion of distance between two probability measures $(\nu^{(i)})_{i=1,2}$ on $\Gamma^2$, which allows for perturbations up to scale $s>0$. 
Namely, we set
$$\cU^s(\nu^{(1)}, \nu^{(2)}) :=\inf_{(Z_{1},Z'_{1}, Z_{2}, Z'_{2})} \Pb(Z'_{1}\neq  Z'_{2})$$
where  $Z_{i}$, $Z'_{i}$ are random variables such that $Z_{i}$ has law $\nu^{(i)}$ and $d(Z_{i}, Z'_{i})\leq s$ almost surely. 
Note that for $s=0$, this recovers the total variation distance. We show the following.

\begin{theorem}\label{Thm:harmonic}
Let $\Gamma$ be a non-elementary word hyperbolic group endowed with a word metric, let $\mu$ be a non-elementary probability measure with finite exponential moment and such that  $\varphi_{\star}\mu$ is centered for some homomorphism $\varphi : \Gamma\rightarrow \R$ which is not identically zero on the support of $\mu$. 

Then for every $\rho, \rho' \in [0, 1]$ with $\rho \neq \rho'$, every $\alpha\in (0, \lambda_{\mu})$, we have  
\[
\cU^{\alpha n}(\pi_{n}^\rho, \pi_{n}^{\rho'}) \to 1 \quad \text{as $n \to \infty$}.
\]
\end{theorem} 

Note that the separation rate is essentially optimal. Indeed for  $\beta>\lambda_{\mu}$, we have $\cU^{\beta n}(\pi_{n}^\rho, \pi_{n}^{\rho'}) \to 0 $ due to the fact that $\pi_{n}^\rho$ and $\pi_{n}^{\rho'}$ are asymptotically supported on the ball of radius  $\beta n$. For $\beta=\lambda_{\mu}$, it is reasonable to expect the mixed behavior $\cU^{\lambda_{\mu} n}(\pi_{n}^\rho, \pi_{n}^{\rho'}) \to 1/2$, at least provided the central limit theorem for the distance between $\id$ and the $\mu$-random walk has non-degenerate limiting variance, see \cite{BenoistQuintCLThyperbolic} for details on that condition.

Theorem \ref{Thm:harmonic} is related to \cite{BenjaminiBrieussel} which exploits different methods to show positive linear separation for some small proportion of the measures $\pi^\rho_n$ and $\mu_n\otimes\mu_n$ in the case where $\Gamma$ is a free group \cite[Theorem 4.2]{BenjaminiBrieussel}.
A generalization of Theorem \ref{Thm:harmonic}, considering groups $\Gamma$ that may not be hyperbolic but act by isometries on a hyperbolic space is also given later in our paper, see Theorem \ref{Thm:harmonic+}.

\subsection{On the proof}

In the proof of Theorems \ref{Thm:intro} and \ref{Thm:harmonic}, 
we study the corresponding harmonic measure $\nu_{\pi^\rho}$ associated to the $\pi^\rho$-random walk on $\Gamma \times \Gamma$. The   probability measure $\nu_{\pi^\rho}$ is defined on the product of Gromov boundaries $(\partial \Gamma)^2$  as the distribution of the almost sure limiting point of $(w_n, w_n^\rho)$ in $(\Gamma \cup \partial \Gamma)^2$ (cf.\ Section \ref{Sec:limit_laws}).
We show that two different values for $\rho \in [0, 1]$ yield two mutually singular harmonic measures $\nu_{\pi^\rho}$. 
This is achieved by examining the homological winding  of a typical pair of geodesic rays selected by $\nu_{\pi^\rho}$, and highlighting distinct behaviors depending on $\rho$. Our approach is inspired by \cite{BenardWinding} where several limit theorems have been obtained for the winding statistics on a hyperbolic group.  Dealing with a \emph{product} of hyperbolic groups requires to develop those tools further. We now give more details.


Let $\f:\Gamma \to \R$ be a homomorphism on a hyperbolic group $\Gamma$. Given $\xi\in \partial \Gamma$, denote by $r_\xi: \N \to \Gamma$  a unit speed geodesic ray
 originating from $\id$ and converging to $\xi$ in $\partial \Gamma$.
Consider a  driving measure $\mu$ on $\Gamma$ which is non-elementary and has finite exponential moment. Let $\nu$ be the associated  harmonic measure on $\partial \Gamma$. Viewing
 $\xi$ as a random variable on $\partial \Gamma$ distributed according to $\nu$, it is established in  \cite[Theorem 1.4]{BenardWinding} that
 the random variables $(\f\circ r_\xi(n))_{n\geq 1}$ satisfy a strong law of large numbers (LLN), 
a central limit theorem (CLT) and 
a law of the iterated logarithm (LIL) as $n$ goes to infinity.
The LLN and the LIL imply that the mean and the variance appearing in those limit theorems depend only on the measure class of $\nu$.
Moreover, it is shown in \cite{BenardWinding}  that those quantities have explicit formulas in terms of the driving measure $\mu$.
Notably,
the formulas provide numerical invariants of all possible $\mu$ such that $\mu\ast \nu=\nu$
(see the more precise discussion in Section \ref{Sec:limit_laws}).

We apply the method to a product group $\Gamma \times \Gamma$ and show a joint version of the LIL in the present setting.
More precisely, provided the pushforward measure $\varphi_{\star}\mu$ is centered and not a Dirac mass, we determine explicitely the set of accumulation points of the sequence
\[
(2n \log \log n)^{-1/2}(\f \circ r_{\xi^{(1)}}(n), \f \circ r_{\xi^{(2)}}(n))
\]
where $(\xi^{(1)}, \xi^{(2)})$ is a pair of points selected by $\nu_{\pi^\rho}$ on $(\partial \Gamma)^2$.  It appears that those accumulation sets are distinct for $\pi^\rho$-random walks corresponding to different $\rho \in [0, 1]$,
thus justifying the mutual singularity of harmonic measures $\nu_{\pi^\rho}$ on $(\partial \Gamma)^2$.
This implies  noise stability with linear separation as in Theorem \ref{Thm:harmonic}  since any  perturbation of $\pi^\rho_n$ at scale $\alpha n$ ($\alpha\in (0, \lambda_{\mu})$) converges weakly to $\nu_{\pi^\rho}$ on $(\Gamma \cup \partial \Gamma)^2$ as $n$ tends to infinity. Theorem \ref{Thm:intro} follows at once from  Theorem \ref{Thm:harmonic}.

The singularity of the harmonic measures raises the following question on their dimensions.
Let $\mu$ be a non-elementary probability measure with finite first moment
on a word hyperbolic group $\Gamma$.
It has been shown that the harmonic measure $\nu_{\pi^\rho}$ is exact dimensional with respect to a quasi-metric in $(\partial \Gamma)^2$ (e.g., the maximum of quasi-metrics in factors) \cite[Theorem 3.1]{NonNS}.
Furthermore,
the Hausdorff dimension is computed as
\[
\dim \nu_{\pi^\rho}=\frac{h(\pi^\rho)}{\lambda_\mu}
\]
where $h(\pi^\rho)$ is the asymptotic entropy of the $\pi^\rho$-random walk. 
 It would be interesting to know as to whether it holds that $\dim \nu_{\pi^\rho}<\dim \nu_{\mu\otimes \mu}$ for all $\rho \in [0, 1)$.

\section{Singularity of harmonic measures and winding statistics}

\subsection{Preliminaries and limit laws on single hyperbolic groups}\label{Sec:limit_laws}

In this section, we fix notations for the rest of the paper and recall some results from \cite{BenardWinding}. For a complementary background on  hyperbolic geometry, see the original paper \cite{GromovHyperbolic} and the survey \cite{Calegari}.

\bigskip

 We let $\Gamma$ be a  hyperbolic (finitely generated) group, endowed with some left-invariant word metric. Given any $\gamma\in \Gamma$, we denote by $|\gamma|$ the distance from $\gamma$ to the identity for the word metric. We let $\partial \Gamma$ denote the Gromov boundary of $\Gamma$ and for every $\xi\in \partial \Gamma$, we choose some geodesic ray $r_{\xi}:\N\rightarrow \Gamma$ such that $r_{\xi}(0)=\id$ and with limiting point $\xi$.  Here geodesic ray means that $|r_{\xi}(k)^{-1}r_{\xi}(l)|=l-k$ for every $k,l\in \N$ with $k\leq l$. On occasion, given $t\in \R^+$, we will write $r_{\xi}(t):=r_{\xi}(\lfloor t  \rfloor)$.

 We let $\mu$ denote a probability measure on $\Gamma$. We set $\Omega =\Gamma^\N$, endowed with the probability measure $\Pr_{\mu}$, which is the pushforward of $\mu^{\otimes \N^{*}}$ by $(\gamma_{i})_{i\geq 1}\mapsto (\gamma_{1}\dots\gamma_{n})_{n\geq 0}$.  A $\Pr_{\mu}$-typical sequence $(w_n)_{n \geq 0}$ represents  a (right) $\mu$-trajectory on $\Gamma$. We assume that $\mu$ has {\bf finite exponential moment}, i.e.\ $\sum_{\gamma\in \Gamma}e^{c| \gamma|}\mu(\gamma)<\infty$ for some $c>0$.
We also suppose that $\mu$ is {\bf non-elementary}. This means 
that  the semigroup generated by the support of $\mu$ contains at least two loxodromic elements with disjoint pairs of fixed points in the Gromov boundary $\partial \Gamma$. In particular $\Gamma$ is non-elementary, i.e.\ it is neither  finite  nor virtually cyclic. Conversely, assuming $\Gamma$ non-elementary, 
any distribution whose support generates $\Gamma$ as a group is non-elementary. As $\mu$ is non-elementary, 
$\Pr_{\mu}$-almost every trajectory $w=(w_{n})\in \Omega$ converges to a point $\xi_{w}$ in $\partial \Gamma$. Moreover, the moment condition on $\mu$ guarantees that the rate of escape to the boundary is linear for the word metric on $\Gamma$. Namely, $\Pr_{\mu}$-almost surely and in $L^1$, we have
\begin{equation}\label{Eq:Kingman}
\frac{1}{n}|w_n|\to \lambda_\mu \quad \text{as $n \to \infty$},
\end{equation}
where $\lambda_\mu >0$ is some deterministic constant (cf.\ \cite[Section 7.1]{BMSS}).

The distribution $\nu_\mu$ of the limiting point $\xi_{w}$ as $w$ varies with law $\Pr_{\mu}$ is called the {\bf harmonic measure} (or the {\bf Furstenberg measure}) associated to $\mu$.
It is the unique probability measure on $\partial \Gamma$ satisfying that $\mu\ast \nu_\mu=\nu_\mu$,
where $\mu\ast \nu_\mu=\sum_{\gamma \in \Gamma}\mu(\gamma)\nu_\mu\circ \gamma^{-1}$. 
The harmonic measure $\nu_\mu$ encodes the asymptotic properties of the $\mu$-random walk trajectories.
It is, however, not clear how much information on $\mu$ can be extracted from (the measure class of) $\nu_\mu$.
In \cite{BenardWinding}, it is shown that $\nu_{\mu}$ conveys information regarding the projection of $\mu$ to the abelianization. Among other results,
one reads the following.

\begin{theorem}[Theorem 1.4 in \cite{BenardWinding}]\label{Thm:Benard}
Keep the above notations and let $\varphi : \Gamma \rightarrow \R$ a  homomorphism. 
\begin{itemize}
\item{\upshape (LLN)} 
For $\nu$-almost every $\xi \in \partial X$, we have as $n \to \infty$, 
\[
\frac{u \circ r_\xi(n)}{n} \to \lambda_{\mu}^{-1}\Eb(\varphi_{\star}\mu).
\]
\item{\upshape (CLT)}
Assume $\Eb(\varphi_{\star}\mu)=0$ and $\xi$ is a random variable distributed according to $\nu_\mu$. Then, as $n \to +\infty$, the sequence of random variables
\[
\frac{u\circ r_\xi(n)}{\sqrt{n}}
\]
converges in law to the centered Gaussian distribution on $\R$ with variance given by $\lambda_{\mu}^{-1}\Var(\varphi_{\star}\mu):=\kappa^2_{\nu_\mu}\in \R^+$ (where $\kappa_{\nu_\mu}\geq 0$).

\item{\upshape (LIL)}
Assume $\Eb(\varphi_{\star}\mu)=0$. For $\nu_\mu$-almost every $\xi \in \partial X$,
the set of accumulation points of the sequence
\[
\frac{u\circ r_\xi(n)}{\sqrt{2n\log \log n}} \quad \text{for $n\geq 3$}
\]
is equal to $[-\kappa_{\nu_\mu}, \kappa_{\nu_\mu}]$.
\end{itemize}
\end{theorem}
In the above statement, $\Eb$ and $\Var$ refer to the mean and variance of a probability measure on $\R$. 

A large deviation principle and a gambler's ruin estimate are also obtained \cite[Theorem 1.4]{BenardWinding}. Note that the CLT and the LIL assume that the pushforward measure $\varphi_{\star}\mu$ is centered. Such limit laws are still available in the non-centered case. However, the involved variance does not have a simple formulation in terms of $\varphi_{\star}\mu$. It may even happen that the limiting Gaussian distribution for geodesic rays is non-degenerate while that of $\varphi_{\star}\mu$ is degenerate,  see \cite[Lemma 4.2]{BenardWinding}.

Concerning the displacements $|w_n|$, we have the corresponding LIL.

\begin{proposition}\label{Prop:LILdistance} Keep the above notations. 
There exists a constant $\sigma_\mu\ge 0$ 
such that for $\Pr_{\mu}$-almost every trajectory $(w_n)_{n\geq 0}$, 
the set of accumulation points of the sequence
\[
\frac{|w_n|-\lambda_\mu n}{\sqrt{2n\log \log n}} \quad \text{for $n\geq 3$},
\]
is equal to $[-\sigma_\mu, \sigma_\mu]$ in $\R$.
\end{proposition}
\proof 
This is a direct application of the LIL for cocycles obtained in \cite[Proposition A.5]{BenardWinding}, which applies thanks to  \cite[Section 3.1]{BenardWinding} and Remark (2) at the beginning of  \cite[Appendix A]{BenardWinding}.  
See also \cite{BjorklundCLT, BenoistQuintCLThyperbolic, ChoiCLT}.
\qed

\subsection{Stopping times and technical lemmas}

We keep the notations $(\Gamma, |.|, r_{\xi}, \mu, \nu_{\mu})$ from the Section \ref{Sec:limit_laws}. We consider $\varphi : \Gamma\rightarrow \R$  such that $\Eb \varphi_{\star}\mu=0$.



The goal of this section is to establish Proposition \ref{Lem:marginal}, 
which tells us that the image under $\varphi$ of a $\nu_{\mu}$-typical geodesic ray on $\Gamma$  behaves like a random walk with centered i.i.d.\ increments at scale $(n \log \log n)^{1/2}$.

\begin{proposition}\label{Lem:marginal}
For $\Pr_\mu$-almost every trajectory $w\in \Omega$, denoting by $\xi_{w} \in \partial \Gamma$ its limiting point, we have as $n \to \infty$,
\[
\frac{1}{\sqrt{n\log \log n}} |\varphi\(r_{\xi_{w}}(\lambda_\mu n)\)-\varphi\(w_n\)| \to 0.
\]
\end{proposition}

For the proof, we first study the sequence of stopping times $\tau_{n}: \Omega\rightarrow \N\cup \{+\infty\}$, where $\tau_{n}$ is the exit time of the centered ball of radius $\lambda_{\mu}n$ in $\Gamma$. More precisely, we define
\[
\tau_n(w) :=\inf\{k \ge 1 \ : \ |w_{k}| \ge \lambda_\mu n\}.
\]
The next lemma gives a first basic estimate on $\tau_{n}$ and on the overshoot.
\begin{lemma}\label{Lem:stopping_time-0}
We have $\Pr_{\mu}$-almost surely $\lim_{n\to +\infty }\tau_n /n= 1$. Moreover, there exists $C>0$ such that for $\Pr_{\mu}$-almost every $w\in \Omega$, for large enough $n$, 
$$| |w_{\tau_n(w)}|-\lambda_\mu n| \leq C\log(n).$$
\end{lemma}

 \begin{proof} The first claim follows from \eqref{Eq:Kingman}. For the second claim, note that the Borel-Cantelli Lemma  and our finite exponential moment assumption on $\mu$ together guarantee that for some $C$ depending on $(\Gamma, |.|, \mu)$, for $\Pr_{\mu}$-almost every $w\in \Omega$, for large enough $n$, we have  $|w_{n-1}^{-1} w_n|\leq C \log n$. On the other hand, by definition of $\tau_n$, we have
\[
|w_{\tau_n(w)}|\ge \lambda_\mu n> |w_{\tau_n(w)-1}|.
\]
Observing that $|w_{\tau_n(w)}|\leq |w_{\tau_{n-1}(w)}|+ |w_{\tau_{n-1}(w)}^{-1} w_{\tau_{n}(w)}|$, the claim follows. 
 \end{proof} 

We deduce that the stopping times $\tau_{n}$ are approximated by $n$ up to an error of  order $(n\log\log n)^{1/2}$.
Here we use the LIL for displacements $|w_{n}|$ cited in Proposition \ref{Prop:LILdistance}.

\begin{lemma}\label{Lem:stopping_time}

There exists a constant $C> 0$ such that for $\Pr_{\mu}$-almost every $w\in \Omega$, for all large enough $n$,
\[
|\tau_n(w)-n|\le C\sqrt{n \log \log n}.
\]
\end{lemma}

\begin{remark} Recall the constant $\sigma_{\mu}$ from Proposition \ref{Prop:LILdistance}. A law of the iterated logarithm with limiting interval $[-\sigma_{\mu}/\lambda_{\mu}, \sigma_{\mu}/\lambda_{\mu}]$ can in fact be obtained for the sequence $(2n \log \log n)^{-1/2}(\tau_n(w)-n)$. The proof below shows all accumulation points belong to this interval. The converse inclusion can be handled by a thickening argument as in \cite[Section 4.4]{BenardWinding}. For conciseness, we do not pursue in this direction. 
\end{remark}

\proof
By Proposition \ref{Prop:LILdistance},
there exists a constant $C> 0$ satisfying that for $\Pr_{\mu}$-almost every $w\in \Omega$, for all large enough $n$,
\[
||w_n|-\lambda_\mu n| \le C\sqrt{n \log \log n}.
\]
Specifying to the subsequence $(\tau_{n})$, we get for large $n$, 
\[
||w_{\tau_n(w)}|-\lambda_\mu \tau_n(w)|\le C\sqrt{\tau_n(w) \log \log \tau_n(w)}.
\]
By Lemma \ref{Lem:stopping_time-0}, we deduce that for $\Pr_{\mu}$-almost every $w\in \Omega$, for all large enough $n$,
\begin{align*}
|\lambda_\mu n-\lambda_\mu \tau_n(w)|
&\le ||w_{\tau_n(w)}|-\lambda_\mu n| +||w_{\tau_n(w)}|-\lambda_\mu \tau_n(w)|\\
&\le (1+o(1))C\sqrt{n\log \log n}.
\end{align*}
This concludes the claim by redefining the constant $C$ to be $2C/\lambda_\mu$.
\qed

\bigskip

We will use the following general lemma to control the partial increments of a sum of i.i.d.\ variables in $\R$ along a short time interval of the order $(n\log\log n)^{1/2}$.

\begin{lemma}\label{Lem:LDP}
Consider a sequence of i.i.d.\ real-random variables $(X_n)_{n\ge 1}$ with finite exponential moment and zero average. Fix $M>0$. 
Set $S_n:=\sum_{i=1}^n X_i$ and  $k_n:=M\sqrt{n\log \log n}$.
Then we have almost surely as $n \to+ \infty$
\[
\frac{1}{k_n}\max\big\{ |S_i-S_n| \ :  i \text{ such that } \ |i-n| \le k_n\big\}\to 0.
\]
\end{lemma}

\proof
Since $\Eb X_i=0$,
the classical large deviation estimate shows that for every $\e>0$ there exist constants $c_\e>0$ and $N_\e$ such that
for all $n \ge N_\e$,
\[
\Pb\big(|S_n| \ge \e n\big) \le e^{-c_\e n}.
\] 
For every $\e>0$,
we have
\begin{align*}
\Pb\Big(\max_{(\log n)^2 \le |i-n| \le k_n} |S_i-S_n| \ge \e k_n\Big)
\le \sum_{(\log n)^2 \le |i-n|\le k_n}\Pb\Big(|S_i-S_n| \ge \e k_n\Big).
\end{align*}
Noting that for $i_1<i_2$, the distribution of $S_{i_2}-S_{i_1}$ and that of $S_{i_2-i_1}$ coincide,
we find that for all $n$ such that $(\log n)^2\ge N_\e$ the right hand side of the above inequality is bounded by 
\[
\sum_{(\log n)^2 \le |i-n|\le k_n}e^{-c_\e |i-n|} \le 2 k_n e^{-c_\e (\log n)^2}=2M\sqrt{n \log \log n}\,e^{-c_\e (\log n)^2}.
\]
Furthermore, since $X_i$ has finite exponential moment, 
for every $\e>0$ there exists a constant $c_\e'>0$ such that $\Pb(|X_i|\ge \e k)\le e^{-c'_\e k}$ for all large enough $k \ge 0$.
Thus for all large enough $n$,
\begin{align*}
&\Pb\Big(\max_{|i-n|\le (\log n)^2}|S_i-S_n| \ge \e k_n\Big)
\le \sum_{|i-n| \le (\log n)^2}\Pb\Big(|S_i-S_n| \ge \e k_n\Big)\\
&\qquad \qquad \qquad \le \sum_{|i-n|\le (\log n)^2}|i-n|e^{-c'_\e k_n/(\log n)^2}\le 2(\log n)^4 e^{-c'_\e k_n/(\log n)^2}\le e^{-n^{1/4}}.
\end{align*}
Combining the above estimates shows that
for every $\e>0$, for all large enough $n$,
\[
\Pb\Big(\max_{|i-n|\le k_n}|S_i-S_n| \ge \e k_n\Big)\le 2M\sqrt{n \log \log n}\,e^{-c_\e (\log n)^2}+e^{-n^{1/4}}.
\]
The right hand side is summable in $n$ for each $\e>0$, whence the claim follows from the Borel-Cantelli lemma.
\qed
\bigskip

We are now able to conclude the proof of the approximation statement Proposition \ref{Lem:marginal}.
\proof
Recall that $\Eb \f(\gamma_i)=0$ and $\mu$ has finite exponential moment. In particular, the  sequence of random variables $(\varphi(w_{n}))_{w\sim \Pr_{\mu}}$ corresponds to the trajectory of a random walk on $\R$ with i.i.d.\ centered increments of finite exponential moment. By Lemma \ref{Lem:LDP}, for each $M>0$, 
we deduce that $\Pr_{\mu}$-almost surely as $n \to \infty$,
\[
\frac{1}{M\sqrt{n\log \log n}}\max\left\{|\f(w_i)-\f(w_n)| \ : \ \text{$i$  such that }|i-n|\le M\sqrt{n\log \log n}\right\}\to 0.
\]
Lemma \ref{Lem:stopping_time} shows that for a constant $C> 0$,  $\Pr_{\mu}$-almost surely, for all large enough $n$,
\[
|\tau_n(w)-n|\le C\sqrt{n\log \log n}.
\]
Therefore, by taking $M=C$, we obtain almost surely
\begin{equation}\label{Eq:Lem:marginal}
\frac{1}{\sqrt{n\log \log n}}|\f(w_{\tau_n(w)})-\f(w_n)| \to 0 \quad \text{as $n \to \infty$}.
\end{equation}
We now relate $w_{\tau_n(w)}$ to the limiting geodesic ray $r_{\xi_{w}}$ associated to $w$.
Using that $\mu$ has finite exponential moment,
there exists a constant $D>0$ such that the following holds:
For $\nu_\mu$-almost every $\xi \in \partial \Gamma$ for all large enough $n$,
\[
|w_n^{-1} r_\xi(|w_{n}|)| \le D \log n.
\] 
(A more general result has been established under a weaker moment assumption \cite[Theorem D]{ChoiCLT}.)
As $\tau_{n}/n\to 1$ a.s., we can specify  to $\tau_{n}$ and get almost surely, for large $n$, 
\[
|w_{\tau_{n}(w)}^{-1} r_\xi(|w_{{\tau_{n}(w)}}|)| \le 2 D \log n.
\] 
On the other hand, by Lemma \ref{Lem:stopping_time-0},
\[
||w_{{\tau_{n}(w)}}| -\lambda_\mu n| \leq C \log n .
\]
Therefore by the triangle inequality, we obtain 
\begin{equation}\label{Eq:Lem:marginal2}
|w_{\tau_{n}(w)}^{-1} r_\xi(\lambda_\mu n)| \le (2 D+C) \log n.
\end{equation}
Since $\varphi$ is Lipschitz, the claim now follows from  the combination of \eqref{Eq:Lem:marginal} and  \eqref{Eq:Lem:marginal2}.
\qed

\subsection{A joint LIL and singularity of harmonic measures}
We keep the notations $(\Gamma, |.|, r_{\xi}, \mu)$ from the Section \ref{Sec:limit_laws} and consider $\varphi : \Gamma\rightarrow \R$  such that $\Eb \varphi_{\star}\mu=0$. We also consider $\pi$ a probability measure on $\Gamma\times \Gamma$ such that both marginals are equal to $\mu$.
We show a joint LIL for the projection under $\varphi$ of a pair of geodesic rays on $\Gamma$ selected randomly according to the harmonic measure of $\pi$. The noise stability for the $\mu$-walk on $\Gamma$  follows.

\bigskip

 We denote $\Pr_{\pi}$ the induced probability measure on the space (double) trajectories $\mb{w}=(w^{(1)}_{n}, w^{(2)}_{n})_{n\geq0}$ (defined similarly to $\Pr_{\mu}$ from Section \ref{Sec:limit_laws}). Note that $\Pr_{\pi}$-almost every trajectory $\mb{w}$ converges in $\partial \Gamma\times \partial \Gamma$. We write $\nu_{\pi}$ for the  distribution of the limiting point, and call it the  harmonic measure for the $\pi$-random walk. We also write $\f\times \f: \Gamma^2\rightarrow \R^2, (\gamma_{1}, \gamma_{2})=(\varphi(\gamma_{1}), \varphi(\gamma_{2}))$

\begin{proposition}\label{Prop:jointLIL}
Let $A_{\nu_\pi}$ be the $2\times 2$ matrix   given by $A_{\nu_\pi}:= \lambda_\mu^{-1} \Cov((\f\times \f)_\star \pi)$. 
Then for $\nu_\pi$-almost every $\mb{\xi}=(\xi^{(1)},\xi^{(2)}) \in \partial \Gamma \times \partial \Gamma$,
the set of accumulation points of the sequence 
\[
\frac{\(\f \circ r_{\xi^{(1)}}(n), \f \circ r_{\xi^{(2)}}(n)\)}{\sqrt{2n\log \log n}} \quad \text{for $n>2$}
\]
is equal to $A_{\nu_\pi}^{1/2}\wbar B(0, 1)$. In particular, $A_{\nu_\pi}$ depends only on the measure class of $\nu_\pi$.
\end{proposition}

In the above,
$\wbar B(0, 1)=\left\{(t_1, t_{2}) \in \R^{2} \ : \ t_1^2+t_{2}^2 \le 1\right\}$, and $\Cov(.)$ refers to the covariance matrix for probability measures on $\R^2$.

\proof[Proof]
Lemma \ref{Lem:marginal} yields that $\Pr_{\pi}$-almost surely, for each $i=1, 2$,
\begin{equation}\label{Eq:Prop:marginal}
\frac{1}{\sqrt{n\log \log n}}|\f(w_n^{(i)})-\f(r_{\xi^{(i)}}(\lambda_\mu n))| \to 0 \quad \text{as $n \to \infty$}.
\end{equation}
By the classical LIL on $\big(\f(w_n^{(1)}), \f(w_n^{(2)})\big)$ for $n\geq 0$,
the following holds:
for the covariance matrix $\SS_\pi:=\Cov((\f\times \f)_\star \pi)$,
for $\Pr_{\pi}$-almost every  $(w_n^{(1)}, w_n^{(2)})$,
the set of accumulation points of the sequence 
\[
\frac{1}{\sqrt{2n \log \log n}}\(\f(w_n^{(1)}), \f(w_n^{(2)})\) \quad \text{for $n\geq 3$},
\]
is equal to $\SS_\pi^{1/2} \wbar B(0, 1)$ in $\R^{2}$.
In view of \eqref{Eq:Prop:marginal}, this estimate still holds with $\f(r_{\xi^{(i)}}(\lambda_\mu n))$ in the place of $\f(w_n^{(i)})$. 
Rescaling by the factor $\lambda_\mu$, we deduce that  for $\nu_\pi$-almost every $\mb{\xi}\in (\partial \Gamma)^2$, the sequence
\[
\frac{1}{\sqrt{2n \log \log n}}\(\f(r_{\xi^{(1)}}(n)), \f(r_{\xi^{(2)}}(n))\) \quad \text{for $n\geq 3$},
\]
has the set of accumulation points equal to $A^{1/2}\wbar B(0, 1)$ where $A:=\lambda_\mu^{-1}\SS_\pi$.
This shows the first claim. 

For the second claim, note we have established that the set $A^{1/2}\wbar B(0, 1)$ depends only on the measure class of $\nu_\pi$ on $(\partial \Gamma)^2$.
Since $A$ is a  positive semi-definite symmetric matrix,
$A$ is uniquely determined by the set $A^{1/2}\wbar B(0, 1)$.
Indeed, if $A_1^{1/2}\wbar B(0, 1)=A_2^{1/2}\wbar B(0, 1)$ for such matrices $A_i$, $i=1, 2$,
then up to restricting to their common image we assume that the $A_i$'s are invertible.
We then get that $A_1^{-1/2}A_2^{1/2}$ is orthogonal. Comparing with its transpose and using that the $A_{i}$'s are symmetric, we get $A_1=A_2$. Therefore letting $A_{\nu_\pi}=\lambda_\mu^{-1}\SS_\pi$
concludes the claim.
\qed

\bigskip
We deduce that  noisy pairs $(w_{n}, w_{n}^{\rho})$ corresponding to different parameters $\rho$ have mutually singular limiting harmonic measures $\nu_{\pi^{\rho}}$.

\begin{corollary} \label{singulairity-nupirho}
Keep the previous notations and assume further that $\f$ is not identically zero on  the support of $\mu$.
Then the harmonic measures $\nu_{\pi^\rho}$ on $\partial \Gamma \times \partial \Gamma$ for $\pi^\rho$-random walks are mutually singular for all $\rho \in [0, 1]$.
\end{corollary}

\proof
A direct computation yields
\[
\Cov((\f\times \f)_\star\pi^\rho)=\Var (\f_{\star}\mu)
\begin{pmatrix}
1 & 1-\rho\\
1-\rho & 1
\end{pmatrix},
\]
where $\Var (\f_{\star} \mu):=\Eb_{\mu}( \f(\gamma)^2)$.
Note that $\Var (\f_{\star} \mu)>0$.
Indeed, otherwise $\f(\gamma)=0$ for all $\gamma$ in the support of $\mu$, contradicting our assumptions. 
Hence,  for $\rho, \rho' \in [0, 1]$
with $\rho \neq \rho'$,
we have $\Cov((\f\times\f)_\star\pi^\rho) \neq \Cov((\f\times \f)_\star \pi^{\rho'})$.
Since by Proposition \ref{Prop:jointLIL} the matrix $A_{\nu_{\pi^\rho}}=\lambda_\mu^{-1}\Cov((\f\times \f)_\star\pi^\rho)$ depends only on the measure class of $\nu_{\pi^\rho}$ for each $\rho \in [0, 1]$,
the  claim follows.
\qed

\bigskip
We still need to convert  limiting singularity into an asymptotic separation estimate. To this end, we need a few preliminary remarks on the behavior of the total variation norm with respect to weak convergence. 

Given a signed Borel measure $\eta$ on a metric space $Z$,
let
\[
\|\eta\|_\TV:=\sup\{|\eta(B)| \ : \ \text{$B$ is Borel in $Z$}\}.
\]
\begin{remark}\label{Rem:TV} Other standard definitions of total variation norm use a variant, which we denote by $\|.\|$. Namely, 
let $\eta=\eta_+ - \eta_-$ be the Hahn-Jordan decomposition, i.e., $\eta_+$ and $\eta_-$ are non-negative Borel measures supported on disjoint Borel sets $Z_+$ and $Z_-$ in $Z$ respectively.
Then  $\|\eta\|:=(\eta_+ + \eta_-)(Z)$.
Note that $\|\eta\| \leq 2\|\eta\|_\TV\leq 2\|\eta\|$. Moreover, if $\eta=\nu^{(1)}-\nu^{(2)}$ for two Borel probability measures $\nu^{(1)}$ and $\nu^{(2)}$, then\footnote{To check this identity, note that
 $\nu^{(1)}(B)-\nu^{(2)}(B)=\nu^{(2)}(Z\smallsetminus B)-\nu^{(1)}(Z\smallsetminus B)$ for every measurable $B\subseteq Z$. Thus $\|\eta\|_\TV=\frac{1}{2}\sup\{|\eta(B_{1})|+|\eta(B_{2})| \ : \ Z=B_{1}\sqcup B_{2}\}= \frac{1}{2}\|\eta\|$.
} we have equality $\|\eta\| = 2\|\eta\|_\TV$. 
\end{remark}

\begin{lemma}\label{Lem:weak} 
If a sequence of Borel probability measures $(\nu_n^{(i)})_{n\geq 0}$ on a metric space $Z$ converges weakly to a Borel probability measure $\nu^{(i)}$ on $Z$ for each $i=1, 2$,
then
\[
\liminf_{n \to \infty}\|\nu_n^{(1)}-\nu_n^{(2)}\|_\TV \ge \|\nu^{(1)}-\nu^{(2)}\|_\TV.
\]
\end{lemma}

\proof
If a sequence of signed Borel measures $\{\eta_n\}_{n \geq 0}$ on $Z$ converges weakly to a signed Borel measure $\eta$,
then
\[
\liminf_{n \to \infty}\|\eta_n\| \ge \|\eta\|
\]
(cf.\ \cite[Theorem 4.8.1]{BogachevWeak}).
The claim holds using Remark \ref{Rem:TV}.
\qed
\bigskip

In order to get noise stability with \emph{linear} separation, we also need to know that convergence toward a given boundary point is stable under a certain range of perturbations.

\begin{lemma} \label{conv-perturb}
Consider two sequences $(\omega_{n})_{n\geq 0}, (\omega'_{n})_{n\geq 0}\in \Gamma^\N$, and a boundary point $\xi \in\partial \Gamma$. Assume $\omega_{n}\to \xi $ and $|\omega_{n}|-|\omega^{-1}_{n}\omega'_{n}|\to + \infty$ as $n$ goes to infinity. Then $\omega'_{n}\to \xi$ as well. 
\end{lemma} 

\begin{proof} Write $(x | y):=\frac{1}{2}\left( |x|+|y|-|x^{-1}y| \right)$ the Gromov product on $\Gamma$ based at the identity.  
We need to check that  $\lim_{k,l\to \infty}(\omega'_{k}| \omega_{l})=\infty$.  By hyperbolicity, we know that  
$$(\omega'_{k} | \omega_{l}) \geq \min\{(\omega'_{k} | \omega_{k}), (\omega_{k} | \omega_{l})\} - O(1).$$
Moreover, each term $(\omega'_{k} | \omega_{k})$, $(\omega_{k} | \omega_{l})$ goes to $\infty$ by assumption, whence the result.
\end{proof}

We are finally able to conclude the proof of Theorem \ref{Thm:harmonic}.

\begin{proof}[Proof of Theorem \ref{Thm:harmonic}]
Note that $\Gamma \cup \partial \Gamma$ admits a compact metrizable topology which is compatible with the notion of convergence to the boundary $\partial \Gamma$, and respects the topologies on $\Gamma$ and $\partial \Gamma$ individually (cf.\ \cite[7.2.M]{GromovHyperbolic}). Moreover, given $\alpha\in (0, \lambda_{\mu})$, the definition of the escape rate \eqref{Eq:Kingman} implies 
$\Pr_{\mu}(|w_{n}|\geq \alpha n)\rightarrow 1$ as $n \to \infty$.  

Equip $(\Gamma \cup \partial \Gamma)^2$ with the product topology. It follows from the previous paragraph and Lemma \ref{conv-perturb} that for every $\rho \in [0, 1]$,  and every sequence of probability measures $\sigma^\rho_{n}$ with $\cU^{\alpha n}(\pi^\rho_n, \sigma^\rho_{n})=0$,
the distributions $\sigma^\rho_{n}$ converge weakly to $\nu_{\pi^\rho}$ in $(\Gamma \cup \partial \Gamma)^2$ as $n \to \infty$.
Lemma \ref{Lem:weak} then implies that for $\rho, \rho' \in [0, 1]$,
\[
\liminf_{n \to \infty}\|\sigma^\rho_{n}- \sigma^{\rho'}_{n}\|_\TV \ge \|\nu_{\pi^\rho}-\nu_{\pi^{\rho'}}\|_\TV.
\]
By Corollary \ref{singulairity-nupirho}, we have  $\|\nu_{\pi^\rho}-\nu_{\pi^{\rho'}}\|_\TV\ge 1$ if $\rho\neq \rho'$.
Noting that $\|\sigma^\rho_{n}-\sigma^{\rho'}_{n}\|_\TV \le 1$ for all $n \geq 0$ (cf.\ Remark \ref{Rem:TV}), we get 
$$\cU^{\alpha n}(\pi^\rho_{n}, \pi^{\rho'}_{n})\to 1 \quad \text{as $n \to \infty$},$$
as required.
\end{proof}

\proof[Proof of Theorem \ref{Thm:intro}] Direct consequence of Theorem \ref{Thm:harmonic} applied to any non zero homomorphism $\varphi : \Gamma \rightarrow \R$.
\qed

\bigskip

As in \cite{BenardWinding}, the proof could have been carried in a slightly more general setting, which allows for a finitely generated group $\Gamma$ that is not hyperbolic, but only acts by isometries on a hyperbolic space. In this context, Theorem \ref{Thm:harmonic} and its proof extend verbatim as follows. 

\begin{theorem}\label{Thm:harmonic+}
Let $(X, d)$ be a proper geodesic hyperbolic metric space endowed with a basepoint $o\in X$. Let $\Gamma$ be a finitely generated group of isometries of $(X, d)$. Let $\mu$ be a non-elementary probability measure on $\Gamma$ with  finite exponential moment. Set $\lambda_{\mu}:=\lim_{n\to \infty} \frac{1}{n}\Eb_{g \sim \mu_{n}}(d(g.o,o))$.

Consider a quasi-Lipschitz  map $u:X \to \R$ satisfying the equivariance relation $u(\gamma.x)=\varphi(\gamma)+u(x)$ for every $\gamma\in \Gamma, x\in X$ and some fixed homomorphism $\f:\Gamma\rightarrow \R$. Assume $\Eb \f_\star\mu=0$ and $\f$ is not identically zero on the support of $\mu$.

Then the harmonic measures $\nu_{\pi^\rho}$ on the product of Gromov boundaries $\partial X\times \partial X$ associated to the $\pi^\rho$-random walks are mutually singular for all $\rho \in [0, 1]$.

In particular, 
it holds that for every $\rho, \rho' \in [0, 1]$ with $\rho \neq \rho'$, every $\alpha\in (0, \lambda_{\mu})$ such that
\[
\cU^{\alpha n}(\pi^\rho_{n}*\delta_{(o,o)}, \,\pi^{\rho'}_{n}*\delta_{(o,o)})
 \to 1 \quad \text{as $n \to \infty$}.
\]
\end{theorem}

Note that the analogous statement to Theorem \ref{Thm:intro} also follows from Theorem \ref{Thm:harmonic+}.

\bibliographystyle{plain}
\bibliography{ns}

\end{document}